\documentclass[twoside,leqno,10pt, A4]{amsart}
\usepackage{amsfonts}
\usepackage{amsmath}
\usepackage{amscd}
\usepackage{amssymb}
\usepackage{amsthm}
\usepackage{amsrefs}
\usepackage{latexsym}
\usepackage{mathrsfs}
\usepackage{bbm}
\usepackage{amscd}
\usepackage{amssymb}
\usepackage{amsthm}
\usepackage{amsrefs}
\usepackage{latexsym}
\usepackage{mathrsfs}
\usepackage{bbm}
\usepackage{enumerate}
\usepackage{graphicx}
\usepackage{color}
\begin{document}

\newtheorem{theorem}[subsection]{Theorem}
\newtheorem{proposition}[subsection]{Proposition}
\newtheorem{lemma}[subsection]{Lemma}
\newtheorem{corollary}[subsection]{Corollary}
\newtheorem{conjecture}[subsection]{Conjecture}
\newtheorem{prop}[subsection]{Proposition}
\newtheorem{defin}[subsection]{Definition}

\numberwithin{equation}{section}
\newcommand{\mr}{\ensuremath{\mathbb R}}
\newcommand{\mc}{\ensuremath{\mathbb C}}
\newcommand{\dif}{\mathrm{d}}
\newcommand{\intz}{\mathbb{Z}}
\newcommand{\ratq}{\mathbb{Q}}
\newcommand{\natn}{\mathbb{N}}
\newcommand{\comc}{\mathbb{C}}
\newcommand{\rear}{\mathbb{R}}
\newcommand{\prip}{\mathbb{P}}
\newcommand{\uph}{\mathbb{H}}
\newcommand{\fief}{\mathbb{F}}
\newcommand{\majorarc}{\mathfrak{M}}
\newcommand{\minorarc}{\mathfrak{m}}
\newcommand{\sings}{\mathfrak{S}}
\newcommand{\fA}{\ensuremath{\mathfrak A}}
\newcommand{\mn}{\ensuremath{\mathbb N}}
\newcommand{\mq}{\ensuremath{\mathbb Q}}
\newcommand{\half}{\tfrac{1}{2}}
\newcommand{\f}{f\times \chi}
\newcommand{\summ}{\mathop{{\sum}^{\star}}}
\newcommand{\chiq}{\chi \bmod q}
\newcommand{\chidb}{\chi \bmod db}
\newcommand{\chid}{\chi \bmod d}
\newcommand{\sym}{\text{sym}^2}
\newcommand{\hhalf}{\tfrac{1}{2}}
\newcommand{\sumstar}{\sideset{}{^*}\sum}
\newcommand{\sumprime}{\sideset{}{'}\sum}
\newcommand{\sumprimeprime}{\sideset{}{''}\sum}
\newcommand{\sumflat}{\sideset{}{^\flat}\sum}
\newcommand{\shortmod}{\ensuremath{\negthickspace \negthickspace \negthickspace \pmod}}
\newcommand{\V}{V\left(\frac{nm}{q^2}\right)}
\newcommand{\sumi}{\mathop{{\sum}^{\dagger}}}
\newcommand{\mz}{\ensuremath{\mathbb Z}}
\newcommand{\leg}[2]{\left(\frac{#1}{#2}\right)}
\newcommand{\muK}{\mu_{\omega}}
\newcommand{\thalf}{\tfrac12}
\newcommand{\lp}{\left(}
\newcommand{\rp}{\right)}
\newcommand{\Lam}{\Lambda_{[i]}}
\newcommand{\lam}{\lambda}
\newcommand{\af}{\mathfrak{a}}
\newcommand{\sw}{S_{[i]}(X,Y;\Phi,\Psi)}
\newcommand{\lz}{\left(}
\newcommand{\pz}{\right)}
\newcommand{\bfrac}[2]{\lz\frac{#1}{#2}\pz}
\newcommand{\odd}{\mathrm{\ primary}}
\newcommand{\even}{\text{ even}}
\newcommand{\res}{\mathrm{Res}}
\newcommand{\sumn}{\sumstar_{(c,1+i)=1}  w\left( \frac {N(c)}X \right)}
\newcommand{\lab}{\left|}
\newcommand{\rab}{\right|}
\newcommand{\Go}{\Gamma_{o}}
\newcommand{\Ge}{\Gamma_{e}}
\newcommand{\M}{\widehat}

\theoremstyle{plain}
\newtheorem{conj}{Conjecture}
\newtheorem{remark}[subsection]{Remark}

\makeatletter
\def\widebreve{\mathpalette\wide@breve}
\def\wide@breve#1#2{\sbox\z@{$#1#2$}%
     \mathop{\vbox{\m@th\ialign{##\crcr
\kern0.08em\brevefill#1{0.8\wd\z@}\crcr\noalign{\nointerlineskip}%
                    $\hss#1#2\hss$\crcr}}}\limits}
\def\brevefill#1#2{$\m@th\sbox\tw@{$#1($}%
  \hss\resizebox{#2}{\wd\tw@}{\rotatebox[origin=c]{90}{\upshape(}}\hss$}
\makeatletter

\title[Ratios conjecture for primitive quadratic Hecke $L$-functions]{Ratios conjecture for primitive quadratic Hecke $L$-functions}

\author[P. Gao]{Peng Gao}
\address{School of Mathematical Sciences, Beihang University, Beijing 100191, China}
\email{penggao@buaa.edu.cn}

\author[L. Zhao]{Liangyi Zhao}
\address{School of Mathematics and Statistics, University of New South Wales, Sydney NSW 2052, Australia}
\email{l.zhao@unsw.edu.au}

\begin{abstract}
 We develop the ratios conjecture with one shift in the numerator and denominator in certain ranges for families of primitive quadratic Hecke $L$-functions of imaginary quadratic number fields with class number one using multiple Dirichlet series under the generalized Riemann hypothesis. We also obtain unconditional asymptotic formulas for the first moments of central values of these families of $L$-functions with error terms of size that is the square root of that of the primary main terms.
\end{abstract}

\maketitle



\section{Introduction}\label{sec 1}

  The $L$-functions ratios conjecture is important in number theory, having many important applications.  These include it use in developing the density conjecture of N. Katz and P. Sarnak \cites{KS1, K&S} on the distribution of zeros near the central point of a family of $L$-functions, the mollified moments of $L$-functions, etc. This conjecture makes predictions on the asymptotic behaviors of the sum of ratios of products of shifted $L$-functions and has its origin in the work of D. W. Farmer \cite{Farmer93} on the shifted moments of the Riemann zeta function. For general $L$-functions, the conjecture is formulated by J. B. Conrey, D. W. Farmer and M. R. Zirnbauer \cite[Section 5]{CFZ}. \newline

 Since the work of H. M. Bui, A. Florea and J. P. Keating \cite{BFK21} on quadratic $L$-functions over function fields, there are now several results available in the literature concerning the ratios conjecture, all valid for certain ranges of the relevant parameters. The first such result over number fields was given by M. \v Cech \cite{Cech1} on families of both general and primitive quadratic Dirichlet $L$-functions under the assumption of the generalized Riemann hypothesis (GRH). \newline

  The work of \v Cech makes uses of the powerful tool of multiple Dirichlet series. Due to an extra functional
equation for the underlying multiple Dirichlet series, the result for the family of general quadratic Dirichlet $L$-functions has a better range of the parameters involved. Following this approach, the authors studied the ratios conjecture for general quadratic Hecke $L$-functions over the Gaussian field $\mq(i)$ in \cite{G&Zhao14} as well as for quadratic twists of modular $L$-functions in \cite{G&Zhao15}. \newline

  It is noted in \cite{G&Zhao14} that investigations on ratios conjecture naturally lead to results concerning the first moment of central values of the corresponding family of $L$-functions, another important subject in number theory. Among the many families of $L$-functions, the primitive ones received much attention. For families of primitive Dirichlet $L$-functions, M. Jutila evaluated asymptotically the first moment in \cite{Jutila}. In \cite{DoHo}, D. Goldfeld and J. Hoffstein initiated the study on the first moment of the same family using method of double Dirichlet series. It is implicit in their work that an asymptotic formula with an error term of size of the square root of the main term holds for the smoothed first moment and this was later established by M. P. Young \cite{Young1} using a recursive argument. In \cite{Gao20}, the first-named author obtained a similar result for the first moment of the family of primitive quadratic Hecke $L$-functions over the Gaussian field using Young's recursive method. \newline

   Motivated by studies of the first moment of families of primitive $L$-functions, we are interested in this paper to develop the ratios conjecture for primitive quadratic Hecke $L$-functions over imaginary quadratic number fields with class number one.  We shall show that, as a consequence of the ratios conjecture, asymptotic formulas hold for the first moments of central values of the corresponding families of $L$-functions with error terms of size of the square root of that of the main terms. \newline

 To state our result, let $K$ be an imaginary quadratic number field of class number one. We shall assume throughout the paper that $K=\mq(\sqrt{d})$ with $d \in \mathcal{S}$, where
\begin{align*}
\mathcal{S} = \{-1, -2, -3, -7,-11,-19,-43,-67,-163 \}.
\end{align*}
  In fact, it is well-known (see \cite[(22.77)]{iwakow}) that we obtain all imaginary quadratic number fields with class number one in the above way. \newline

   We write $\chi^{(m)}=\leg {m}{\cdot}$ and $\chi_{m} =\leg {\cdot}{m}$ for the quadratic residue symbols to be defined Section \ref{sec2.4}. We define $c_K=(1+i)^5$ when $d=-1$, $c_K=4\sqrt{-2}$ when $d=-2$ and $c_K=8$ for the other $d$'s in $\mathcal{S}$.  Let $\mathcal{O}_K, U_K$ and $D_K$ denote the ring of integers, the group of units and the discriminant of $K$, respectively. We also write $|U_K|$ for the cardinality of $U_K$.  An element $c \in \mathcal O_K$ is said to be square-free if the ideal $(c)$ is not divisible by the square of any prime ideal.  It is shown in \cite[Section 2.1]{G&Zhao4} and \cite[Section 2.2]{G&Zhao2022-4} that $\chi^{(c_Kc)}$ is a primitive quadratic character of trivial infinite type for any square-free $c$. \newline

Let $L(s, \chi)$ stand for the $L$-function attached to any Hecke character $\chi$ and $\zeta_K(s)$ for the Dedekind zeta function of $K$.  We also use the notation $L^{(c)}(s, \chi_m)$ for the Euler product defining $L(s, \chi_m)$ but omitting those primes dividing $c$. Moreover, let $r_K$ denote the residue of $\zeta_K(s)$ at $s = 1$.  We reserve the letter $\varpi$ for a prime element in $\mathcal O_K$, by which we mean that the ideal $(\varpi)$ generated by
$\varpi$ is a prime ideal. Let also $N(n)$ stand for the norm of any $n \in \mathcal{O}_K$.  In the sequel, $\varepsilon$ always, as is standard, denotes a small positive real number which may not be the same at each occurrence. \newline

Our main result in this paper investigates the ratios conjecture with one shift in the numerator and denominator for the family of primitive quadratic Hecke $L$-functions of $K$.
\begin{theorem} \label{Theorem for all characters}
With the notation as above and assuming the truth of GRH, let $K=\mq(\sqrt{-d})$ with $d \in \mathcal  S$.  Suppose that $w(t)$ is a non-negative Schwartz function and $\widehat w(s)$ its Mellin transform.  We set
\begin{equation} \label{Nab}
			E(\alpha,\beta)=\max\left\{\frac 12, 1-\Re(\alpha)-\Re(\beta), 1-\frac {\Re(\alpha)}{2}-\frac {\Re(\beta)}{2}, \frac 12-\frac {\Re(\alpha)}{2}, \frac 12-\Re(\alpha) \right\}.
\end{equation}
and, for any $n \in \mathcal O_K$, 
\begin{align} \label{adef}
	a(n)=\prod_{\varpi|n}\Big(1+ \frac 1{N(\varpi)}\Big)^{-1}.
\end{align}
  Then we have for $0<|\Re(\alpha)|<1/2$ and $\Re(\beta)>0$,
\begin{align}
\label{Asymptotic for ratios of all characters}
\begin{split}	
\sumstar_{(c,2)=1} & \frac{L(\frac{1}{2}+\alpha,\chi^{(c_Kc)})}{L(\frac{1}{2}+\beta,\chi^{(c_Kc)})} w\left( \frac {N(c)}X \right) \\
 &=  X\M w(1)\frac {r_K |U_K|^2a(2)}{\zeta_K(2)} \frac{\zeta_K^{(2)}(1+2\alpha)}{\zeta_K^{(2)}(1+\alpha+\beta)}
P(\tfrac 12+\alpha, \tfrac 12+\beta)\\
& \hspace*{1cm}  +X^{1-\alpha}\M w(1-\alpha) \frac{(2\pi)^{2\alpha}}{(|D_K|N(c_K))^{\alpha}} \frac {\Gamma(\tfrac{1}{2}-\alpha)}{\Gamma (\tfrac{1}{2}+\alpha)}\frac {r_K |U_K|^2a(2)}{\zeta_K(2)}  \frac{\zeta_K^{(2)}(1-2\alpha)}{\zeta_K^{(2)}(1-\alpha+\beta)}P(\tfrac 12-\alpha,\tfrac 12+\beta) \\
& \hspace*{1cm}  +O\lz(1+|\alpha|)^{5/2+\varepsilon}(1+|\beta|)^{\varepsilon} X^{E(\alpha,\beta)+\varepsilon}\pz,
\end{split}
\end{align}
 where $\sum^*$ means that the sum is restricted to square-free elements $c$ of $\mathcal{O}_K $ and
\begin{equation}
\label{Pwz}
		P(w,z)=\prod_{(\varpi,2)=1}\lz1+\frac{1-N(\varpi)^{z-w}}{(N(\varpi)+1)(N(\varpi)^{z+w}-1)}\pz.
\end{equation} 	
\end{theorem}

  An evaluation from the ratios conjecture on the left-hand side of \eqref{Asymptotic for ratios of all characters} can be derived following the treatments given in \cite[Section 5]{G&Zhao2022-2}. One sees that our result above is consistent with the result therefrom, except that this approach predicts that \eqref{Asymptotic for ratios of all characters} holds uniformly for $|\Re(\alpha)|< 1/4$, $(\log X)^{-1} \ll \Re(\beta) < 1/4$ and $\Im(\alpha)$, $\Im(\beta) \ll X^{1-\varepsilon}$ with an error term $O(X^{1/2+\varepsilon})$. The advantage of \eqref{Asymptotic for ratios of all characters} here chiefly lies in the absence of any constraint on imaginary parts of $\alpha$ and $\beta$.  Moreover, we do obtain an error term of size $O(X^{1/2+\varepsilon})$ unconditionally when $\beta$ is large enough. In fact, one checks by going through our proof of Theorem \ref{Theorem for all characters} in Section \ref{sec: proof of main result} that the only place we need to assume GRH is \eqref{Lindelof}, to estimate the size of $1/L(z,\chi^{(c_Kc)})$. However, when $\Re(z)$ is larger than one, then  $1/L(z,\chi^{(c_Kc)})$ is $O(1)$ unconditionally. Replacing this estimation everywhere in the rest of the proof of Theorem \ref{Theorem for all characters}, and then consider the case $\beta \rightarrow \infty$, we immediately deduce from  \eqref{Asymptotic for ratios of all characters} the following unconditional result on the smoothed first moment of values of quadratic Hecke $L$-functions.
\begin{theorem}
\label{Thmfirstmoment}
Using the notation as above and assuming the truth of GRH, let $K=\mq(\sqrt{-d})$ with $d \in \mathcal  S$. We have for $0<|\Re(\alpha)|<1/2$,
\begin{align}
\label{Asymptotic for first moment}
\begin{split}
\sumstar_{(c,2)=1} & L(\tfrac{1}{2}+\alpha,\chi^{(c_Kc)})w\left( \frac {N(c)}X \right) \\
=&  X\M w(1)\frac {r_K |U_K|^2a(2)}{\zeta_K(2)} \zeta_K^{(2)}(1+2\alpha)
P(\tfrac 12+\alpha)\\
& \hspace*{1cm}  +X^{1-\alpha}\M w(1-\alpha)(|D_K|N(c_K))^{-\alpha}(2\pi)^{2\alpha}\frac {\Gamma(\tfrac{1}{2}-\alpha)}{\Gamma (\tfrac{1}{2}+\alpha)}\frac {r_K |U_K|^2a(2)}{\zeta_K(2)}  \zeta_K^{(2)}(1-2\alpha)P(\tfrac 12-\alpha) \\
& \hspace*{1cm}  +O\lz(1+|\alpha|)^{5/2+\varepsilon} X^{E(\alpha)+\varepsilon}\pz,
\end{split}
\end{align}
  where
\begin{align*}
 E(\alpha)=\lim_{\beta \rightarrow \infty}E(\alpha,\beta)=\max\left\{\frac 12, \frac 12-\frac {\Re(\alpha)}{2}, \frac 12-\Re(\alpha) \right\}, \;
	P(w) =\lim_{z \rightarrow \infty}P(w,z)=\prod_{(\varpi,2)=1}\lz1-\frac{1}{(N(\varpi)+1)N(\varpi)^{2w}}\pz.
\end{align*} 	
\end{theorem}

 Note that the $O$-term in \eqref{Asymptotic for first moment} is uniform in $\alpha$, we can further take the limit $\alpha \rightarrow 0^+$ and deduce the following asymptotic formula for the smoothed first moment of central values of quadratic Hecke $L$-functions.
\begin{corollary}
\label{Thmfirstmomentatcentral}
		With the notation as above, let $K=\mq(\sqrt{-d})$ for $d \in \mathcal  S$.  We have,
\begin{align*}
\begin{split}			
	& 	\sumstar_{(c,2)=1}  L(\tfrac{1}{2}+\alpha,\chi^{(c_Kc)})w\left( \frac {N(c)}X \right)  = XQ_K(\log X)+O\lz  X^{1/2+\varepsilon}\pz.
\end{split}
\end{align*}
  where $Q_K$ is a linear polynomial whose coefficients depend only $K$, $\M w(1)$ and $\M w'(1)$.
\end{corollary}

  As another application of Theorem \ref{Theorem for all characters}, we differentiate with respect to $\alpha$ in \eqref{Asymptotic for ratios of all characters} and then set $\alpha=\beta=r$ to obtain an asymptotic formula concerning the smoothed first moment of $L'(\frac{1}{2}+r,\chi^{(c_Kc)})/L(\frac{1}{2}+r,\chi^{(c_Kc)})$ averaged over odd, square-free $c$.
\begin{theorem}
\label{Theorem for log derivatives}
	With the notation as above and assuming the truth of GRH. Let $K=\mq(\sqrt{-d})$ with $d \in \mathcal  S$. We have for $0<\varepsilon<\Re(r)<1/2$,
\begin{align}
\label{Sum of L'/L with removed 2-factors}
\begin{split}
\sumstar_{(c,2)=1} & \frac{L'(\frac{1}{2}+r,\chi^{(c_Kc)})}{L(\frac{1}{2}+r,\chi^{(c_Kc)})} w\left( \frac {N(c)}X \right) \\
			&= X\M w(1)\frac {r_K |U_K|^2a(2)}{\zeta_K(2)} \lz\frac{(\zeta^{(2)}_K(1+2r))'}{\zeta^{(2)}_K(1+2r)}+\sum_{(\varpi, 2)=1}\frac{\log N(\varpi)}{N(\varpi)(N(\varpi)^{1+2r}-1)}\pz \\
			& \hspace*{1cm} - X^{1-r}\M w(1-r)(|D_K|N(c_K))^{-r}(2\pi)^{2r}\frac {\Gamma(1/2-r)}{\Gamma (1/2+r)}\frac {|U_K|^2a(2)}{\zeta^{(2)}_K(2-2r)}+O((1+|r|)^{5/2+\varepsilon}X^{1-2r+\varepsilon}).
\end{split}
\end{align}
\end{theorem}

	Theorem \ref{Theorem for all characters} will be established using the approach in the proof of \cite[Theorem 1.1]{Cech1} with a refinement on the arguments there.  More precisely, we note that the reciprocal of $\zeta_{K}(2s)$ appears in the expression for $A_K(s,w,z)$ in \eqref{Aswd}. A similar expression involving $(\zeta(2s))^{-1}$ also emerges in the proof of \cite[Theorem 1.1]{Cech1}, where $\zeta(s)$ is the Riemann zeta function. In the proof of \cite[Theorem 1.1]{Cech1}, the factor  $(\zeta(2s))^{-1}$ is treated directly so that one needs to restrict $s$ to the region $\Re(s)>1/2$. Our idea here is to estimate the product of $\zeta_{K}(2s)$ and $A_K(s,w,z)$ instead (see \eqref{Aswboundsumm} below), this allows us to obtain the analytic property of $A_K(s,w,z)$ in a larger region of $s$, which leads to improvements on the error term in \eqref{Asymptotic for ratios of all characters} over \cite[Theorem 1.1]{Cech1}. \newline

  We may also apply Theorem \ref{Theorem for log derivatives} to compute the one-level density of low-lying zeros of the corresponding families of quadratic Hecke $L$-functions, as in \cite[Corollary 1.5]{Cech1} for the family of quadratic Dirichlet $L$-functions using \cite[Theorem 1.4]{Cech1}. However, one checks that in this case we may only do so for test functions whose Fourier transform being supported in $(-1, 1)$, which is inferior to what may be obtained using other methods.  Thus we shall not go in this direction here.

\section{Preliminaries}
\label{sec 2}

\subsection{Imaginary quadratic number fields}
\label{sect: Kronecker}

Recall that for an arbitrary number field $K$,  $\mathcal{O}_K, U_K$ and $D_K$ stand for the ring of integers, the group of units and the discriminant of $K$, respectively.  We say an element $c \in \mathcal O_K$ is odd if $(c, 2)=1$. \newline

In the rest of this section, let $K$ be an imaginary quadratic number field. The following facts concerning an imaginary quadratic number field $K$ can be found in \cite[Section 3.8]{iwakow}. \newline

  The ring of integers $\mathcal{O}_K$ is a free $\mz$ module (see \cite[Section 3.8]{iwakow}) such that $\mathcal{O}_K=\mz+\omega_K \mz$, where
\begin{align*}
   \omega_K & =\begin{cases}
\displaystyle     \frac {1+\sqrt{d}}{2} \qquad & d \equiv 1 \pmod 4, \\ \\
     \sqrt{d} \qquad & d \equiv 2, 3 \pmod 4.
    \end{cases}
\end{align*}
 Note that when $K=\mq(\sqrt{-3})$, we also have $\mathcal{O}_K=\mz+\omega^2_K \mz$.  The group of units are given by
\begin{align}
\label{Uk}
   U_K & =\begin{cases}
     \{ \pm 1 , \pm i \} \qquad & d=-1, \\
     \{ \pm 1, \pm \omega_K, \pm \omega_K^2 \} \qquad & d=-3,  \\
     \{ \pm 1  \} \qquad & \text{other $d$}.
    \end{cases}
\end{align}

  Moreover, the discriminants are given by
\begin{align*}
   D_K & =\begin{cases}
     d \qquad & \text{if $d \equiv 1 \pmod 4$}, \\
     4d \qquad & \text{if $d \equiv 2, 3 \pmod 4$}.
    \end{cases}
\end{align*}

If we write $n=a+b\omega_K$ with $a, b \in \mz$, then
\begin{align}
\label{Norm}
 N(n)=\begin{cases}
   \displaystyle  a^2+ab+b^2\frac {1-d}{4} \qquad & d \equiv 1 \pmod 4,  \\
   \displaystyle  a^2-db^2 \qquad &  d \equiv 2, 3 \pmod 4.
    \end{cases}
\end{align}

   Denote the Kronecker symbol in $\mz$ by $\leg {\cdot}{\cdot}_{\mz}$.  A rational prime ideal $(p) \in \mz$ in $\mathcal O_K$ factors in the following way:
\begin{align*}
    & \leg {D_K}{p}_{\mz} =0, \text{then $p$ ramifies}, \ (p)=\mathfrak{p}^2 \quad \text{with} \quad N(\mathfrak{p})=p, \\
    & \leg {D_K}{p}_{\mz} =1, \text{then $p$ splits}, \ (p)=\mathfrak{p}\overline{\mathfrak{p}} \quad \text{with} \quad  \mathfrak{p} \neq \overline{\mathfrak{p}}
    \quad \text{and} \quad  N(\mathfrak{p})=p, \\
    & \leg {D_K}{p}_{\mz}=-1, \text{then $p$ is inert}, \ (p)=\mathfrak{p} \quad \text{with} \quad  N(\mathfrak{p})=p^2.
\end{align*}

In particular, the rational prime ideal $(2) \in \mq$ splits if $d \equiv 1 \pmod 8$, is inert if $d \equiv 5 \pmod 8$ and ramifies for $d=-1$, $-2$.

\subsection{Primary Elements}
\label{sect: PrimElem}

   It is well-known that when a number field $K$ is of class number one, every ideal of $\mathcal O_K$ is principal, so that one may fix
a unique generator for each non-zero ideal. In this paper, we are interested in ideals that are co-prime to $(2)$ and in this section, we determine for each such ideal a unique generator which we call primary elements. Our choice for these primary elements is based on a result of F. Lemmermeyer \cite[Theorem 12.17]{Lemmermeyer05}, in order to ensure the validity of a proper form of the quadratic reciprocity law, Lemma~\ref{Quadrecgenprim}, for these elements.  \newline

  For $K=\mq(i)$, $(1+i)$ is the only ideal above the rational ideal $(2) \in \mz$.  We define for any $n \in \mathcal O_K$ with $(n, 2)=1$ to be primary if and only if $n \equiv 1 \pmod {(1+i)^3}$ by observing that the group $\left (\mathcal{O}_K / ((1+i)^3) \right )^{\times}$ is isomorphic to $U_K$.  Here and in what follows,  for any $q \in \mathcal{O}_K$, let $G_q := \left( \mathcal{O}_K / (q) \right )^{\times}$ be the group of reduced residue classes modulo $q$, i.e. the multiplicative group of invertible elements in $\mathcal{O}_K / (q)$. For any group $G$, let $G^2$ be the subgroup of $G$ consisting of square elements of $G$.  \newline

  In what follows we consider the case for $K=\mq(\sqrt{d})$ with $d \in \mathcal S$ and $d \neq -1$. Note that we have either $d=4k+1$ for some $k \in \intz$ or $d=-2$. Recall that our choice for the primary elements is based on \cite[Theorem 12.17]{Lemmermeyer05}, which requires the determination of the quotient group $G_4/G^2_4$. We therefore begin by examining the structure of $G_4$. \newline

  We first consider the case $d=4k+1$.  Note that by our discussions in Section \ref{sect: Kronecker}, the rational ideal $(2)$ splits in $\mathcal O_K$ when $2|k$ and is inert otherwise. It follows that
\begin{align}
\label{residuemod2}
\begin{split}
   G_2=
\begin{cases}
 \{1, -\omega_K, 1-\omega_K\}, & 2 \nmid k, \\
  \{1 \}, & 2|k.
\end{cases}
\end{split}
\end{align}
As $|G_2|$ is odd in either case above, the natural homomorphism
\begin{align*}
\begin{split}
   G_2 \rightarrow G^2_2 : g \mapsto g^2,
\end{split}
\end{align*}
  is injective and hence an isomorphism. Now observe that
\begin{align*}
\begin{split}
  G_4= G_2+ 2l, \quad l \in \{ 0, 1, \omega_K, 1+\omega_K\}.
\end{split}
\end{align*}
As $(a+2l)^2 \equiv a^2 \pmod 4$, we see that $G^2_4=G^2_2 \cong G_2$. \newline

  Direct computation shows that
\begin{align}
\label{G2square}
\begin{split}
   G^2_4=G^2_2=
\begin{cases}
 \{1, \omega^2_K=k+\omega_K, (1+\omega_K)^2=k+1-\omega_K \}, & 2 \nmid k, \\
  \{1 \}, & 2|k.
\end{cases}
\end{split}
\end{align}
 Here we remark that the reason we choose $-\omega_K$ instead of $\omega_K$ in the set \eqref{residuemod2} defining $G_2$ is to ensure that when $d=-3$, we have $G^2_2=U^2_K$. \newline

   Consider the exact sequence
\begin{align}
\label{G4G2}
\begin{split}
  0 \rightarrow \ker{\sigma}\rightarrow G_4 \xrightarrow{\sigma} G_2 \rightarrow 0,
\end{split}
\end{align}
 where $\sigma(g)=g \pmod 2$ for any $g \in G_4$. It follows that $G_4 \cong \ker{\sigma} \times G_2 \cong \ker{\sigma} \times G^2_4$. As $\ker{\sigma}$ contains four elements that form a representation of all the residue classes modulo $4$ that are congruent to $1$ modulo $2$, we see that
\begin{align}
\label{classonemod4}
\begin{split}
   \ker{\sigma} = \{ \pm 1, \pm (1+2\omega_K) \} \pmod 4  \cong  <- 1> \times <1+2\omega_K> \pmod 4,
\end{split}
\end{align}
where $<a>$ denotes the cyclic group generated by $a$. \newline

  We then conclude that when $d=4k+1$, we have
\begin{align}
\label{G4}
\begin{split}
   G_4 \cong G^2_4 \times <- 1> \times <1+2\omega_K> \pmod 4.
\end{split}
\end{align}

  When $d =- 2 \pmod 4$, the rational ideal $(2)$ ramifies in $\mathcal O_K$ and we have
\begin{align*}
\begin{split}
   G_2=\{1, 1+\omega_K \}.
\end{split}
\end{align*}
  Note that in this case $G^2_2=\{1\}$, but our discussions above imply that
\begin{align*}
\begin{split}
   G^2_4=\{1, (1+\omega_K)^2 \}=\{1, -1+2\omega_K \} \pmod 4.
\end{split}
\end{align*}

  One verifies directly that when $d=-2$,
\begin{align}
\label{residuemod4d2mod4}
\begin{split}
   G_4/G^2_4 = \{ \pm 1, \pm (1+\omega_K) \}=\{ \pm 1\} \times \{1, -(1+\omega_K)\} \pmod 4.
\end{split}
\end{align}

  We further note that if $d \neq -3$, we have $U_K \cong <-1>$ by \eqref{Uk} and we have $G^2_4 \cong U^2_K$ for $d=-3$. Moreover, note that we have
$U_K \cong U^2_K \times <-1>$. It follows from these observations that we can apply the isomorphism in \eqref{G4} and the quotient group given in \eqref{residuemod4d2mod4} to write $G_4$ as a set by
\begin{align*}
\begin{split}
   G_4 =
\begin{cases}
 U_K \times G^2_4 \times <1+2\omega_K> \pmod 4, & d \neq -2, -3, \\
 U_K \times G^2_4 \times \{1, -(1+\omega_K)\} \pmod 4, & d=-2, \\
 U_K \times <1+2\omega_K> \pmod 4, & d=-3.
\end{cases}
\end{split}
\end{align*}
  We then define the primary elements to be the ones that are congruent to the group $G_4/U_K$ modulo $4$ using the above representation. More precisely, the primary elements are the ones 
\begin{align}
\label{primary}
\begin{split}
   \equiv
\begin{cases}
 G^2_4 \times <1+2\omega_K> \pmod 4, & d \neq -2, -3,  \\
  G^2_4 \times \{1, -(1+\omega_K)\}  \pmod 4, & d = -2, \\
 <1+2\omega_K> \pmod 4, & d=-3.
\end{cases}
\end{split}
\end{align}
  It is then easy to see that each ideal co-prime to $2$ has a unique primary generator. Moreover, as $G_4/U_K$ is a group, we see that primary elements are closed under multiplication.

  We close this section by pointing out a relationship between the notion of primary elements for $d=-3$ with that of $E$-primary introduced in \cite[Section 7.3]{Lemmermeyer}.  Recall that any $n=a+b\omega^2_K \in \mathcal O_K$ with $a, b \in \mz$ and $(n, 6)=1$ is defined to be $E$-primary if $n \equiv \pm 1 \pmod 3$ and satisfies
\begin{align}
\label{cubicE}
\begin{split}
   & a+b \equiv 1 \pmod 4, \quad  \text{if} \quad 2 | b,  \\
   & b \equiv 1 \pmod 4, \quad  \text{if} \quad 2 | a,  \\
   & a \equiv 3 \pmod 4, \quad  \text{if} \quad 2 \nmid ab.
\end{split}
\end{align}
    Now, for any primary $n=a+b\omega_K \in \mathcal O_K$ with $a, b \in \mz$ such that $(n,6)=1$, there is a unique $u \in U^2_K=\{ 1, \omega_K, \omega^2_K\}$ such that $un \equiv \pm 1 \pmod 3$. Writing $un$ in the form $c+d\omega^2_K$ using the observation that $\omega_K=-1-\omega^2_K$, one checks using \eqref{cubicE} directly that $un$ is $E$-primary. \newline

\subsection{Quadratic Hecke characters and quadratic Gauss sums}
\label{sec2.4}
  Let $K$ be any imaginary quadratic number field. We say an element $n \in \mathcal{O}_K$ is odd if $(n,2)=1$. For an odd prime $\varpi \in \mathcal{O}_{K}$, the quadratic symbol $\leg {\cdot}{\varpi}$ is defined for $a \in \mathcal{O}_{K}$, $(a, \varpi)=1$ by $\leg{a}{\varpi} \equiv a^{(N(\varpi)-1)/2} \pmod{\varpi}$, with $\leg{a}{\varpi} \in \{\pm 1 \}$.  If $\varpi | a$, we define $\leg{a}{\varpi} =0$.  We then extend the quadratic symbol to $\leg {\cdot}{n}$ for any odd $n$ multiplicatively. We further define $\leg {\cdot}{c}=1$ for $c \in U_K$. \newline

 We note the following quadratic reciprocity law concerning primary elements.
\begin{lemma}
\label{Quadrecgenprim}
   Let $K=\mq(\sqrt{d})$ with $d \in \mathcal{S}$. For any co-prime primary elements $n,m$ with $(nm, 2)=1$, we have
\begin{align}
\label{quadreciKprim}
    \leg {n}{m}\leg{m}{n}=(-1)^{(N(n)-1)/2\cdot (N(m)-1)/2}.
\end{align}
\end{lemma}
\begin{proof}
  The assertion of the lemma is in \cite[Lemma 2.4]{G&Zhao16} for $d \in S \setminus \{-1, -3\}$ and one checks the arguments there also gives that \eqref{quadreciKprim} holds for $d=-3$.  If $d=-1$, upon noting that $N(n)\equiv 1 \pmod 4$ for any odd $n$, we see that \eqref{quadreciKprim} follows from \cite[(2.1)]{G&Zhao4}. This completes the proof.
\end{proof}

   We also note that it follows from the definition that for any odd $n \in \mathcal O_K$,
\begin{align}
\label{2.051}
  \leg {-1}{n}=(-1)^{(N(n)-1)/2}.
\end{align}
  Moreover, the following supplementary laws hold for primary odd $n=a+bi \in \mathcal O_K$ with $a, b \in \mz$,
\begin{align}
\label{supprulequadtic}
\begin{split}
  & \leg {i}{n}=(-1)^{(1-a)/2}, \quad n=a+bi \in \mathcal O_K, a, b \in \mz, K=\mq(i), \\
  & \leg {1-\omega_K}{n}=1, \quad K=\mq(\sqrt{-3}).
\end{split}
\end{align}
 The first identity above follows from  \cite[Lemma 8.2.1]{BEW} and the second follows by noting that $(1-\omega_K)^3=1$. \newline

   We follow the nomenclature  of \cite[Section 3.8]{iwakow} to define a Dirichlet character $\chi$ modulo $(q) \neq (0)$ to be a homomorphism
\begin{align*}
  \chi: \left (\mathcal{O}_K / (q) \right )^{\times}  \rightarrow S^1 :=\{ z \in \mc :  |z|=1 \}.
\end{align*}
  We say that $\chi$ is primitive modulo $(q)$ if it does not factor through $\left (\mathcal{O}_K / (q') \right )^{\times}$ for any divisor $q'$ of $q$ with $N(q')<N(q)$. \newline

Suppose that $\chi(u)=1$ for any $u \in U_K$.  Then $\chi$ may be regarded as defined on ideals of $\mathcal O_K$ since every ideal is principal. In this case, we say that such a character $\chi$ is a Hecke character modulo $(q)$ of trivial infinite type. We also say such a Hecke character is primitive if it is primitive as a Dirichlet character. We say that $\chi$ is a Hecke character modulo $q$ instead of modulo $(q)$ if there is ambiguity. \newline

    For any Hecke character $\chi$ modulo $q$ of trivial infinite type and any $k \in \mathcal O_K$, we define the associated Gauss sum $g_K(k, \chi)$ by
\begin{align*}
 g_K(k,\chi) = \sum_{x \shortmod{q}} \chi(x) \widetilde{e}_K\leg{kx}{q}, \quad \mbox{where} \quad \widetilde{e}_K(z) =\exp \left( 2\pi i  \left( \frac {z}{\sqrt{D_K}} -
\frac {\overline{z}}{\sqrt{D_K}} \right) \right).
\end{align*}

We write $g_K(\chi)$ for $g_K(1,\chi)$ and note that our definition above for $g_K(k,\chi)$ is independent of the choice of a generator for $(q)$.
We define similarly $g_K(k, \chi_n)$, $g_K(\chi_n)$. The following result evaluates $g_K(\chi_{n})$ for any primary $n$.
\begin{lemma}
\label{Gausssum}
    Let $K=\mq(\sqrt{d})$ with $d \in \mathcal S$ and $n$ a primary element in $\mathcal{O}_K$. Then
\begin{align}
\label{gvalue}
   g_K(\chi_{n})= \begin{cases}
    \sqrt{N(n)} \qquad & \text{for all} \; d, \; N(n) \equiv 1 \pmod 4, \\
    -i \sqrt{N(n)} \qquad & \text{for} \; d \neq -2,-7, \; N(n) \equiv -1 \pmod 4,\\
     i \sqrt{N(n)} \qquad & \text{for} \; d = -2, -7, \; N(n) \equiv -1 \pmod 4.
\end{cases}
\end{align}
\end{lemma}
\begin{proof}
  We consider \eqref{gvalue} for the case of $n$ being a primary prime first.  Note that this has been established for the case $d=-1$ in \cite[Lemma 2.2]{G&Zhao4} for primary primes. In the rest of the proof, we assume that $d \neq -1$. We note that it is shown in \cite[Lemma 2.6]{G&Zhao16} that for any $P$-primary prime $\varpi$ that is co-prime to $2D_K$ and also to $(1-d)/4$ when $d \equiv 1 \pmod 4$,
\begin{align}
\label{gprimevalue}
   g_K(\chi_{\varpi})= \begin{cases}
    \sqrt{N(\varpi)} \qquad & \text{for all} \; d, \; N(\varpi) \equiv 1 \pmod 4, \\
    -i \sqrt{N(\varpi)} \qquad & \text{for} \; d \neq -2,-7, \; N(\varpi) \equiv -1 \pmod 4.
\end{cases}
\end{align}
  Here we recall that when $d \neq -2$,  an odd element $n=a+b\omega_K$ with $a, b \in \mz$ is called $P$-primary if $b \equiv 1 \pmod 4$ or $a+b(1-d)/4 \equiv -1 \pmod 4$ when $(b, 2) \neq 1$. When $d = -2$,  an odd element $n=a+b\omega_K$ with $a, b \in \mz, b \neq 0$ is called $P$-primary when $b=2^kb'$ with $k, b' \in \mz, k \geq 0, b' \equiv 1 \pmod 4$. The arguments in the proof of \cite[Lemma 2.6]{G&Zhao16} in fact imply that this is indeed true for any $P$-primary prime. To see this, first note that when a prime $\varpi$ satisfies $N(\varpi)=p$ with $p$ a rational prime.  Then we write $\varpi=a+b\omega_K$ to see via \eqref{Norm} that
\begin{align*}
   p & =\begin{cases}
    \displaystyle a^2+ab+b^2\frac {1-d}{4}, \qquad & \text{if} \; d \equiv 1 \pmod 4, \\
     \displaystyle a^2-db^2,  \qquad & \text{if}  \; d \equiv 2, 3 \pmod 4.
    \end{cases}
\end{align*}
   The above then implies that $(ab,p)=1$. This is clear for the case $d \equiv 2, 3 \pmod 4$ since the above expression implies that $a^2-db^2>p$ when $(p, ab)>1$. On the other hand, when $d \equiv 1 \pmod 4$, we have $(1-d)/4 \geq 1$ so that $a^2+ab+b^2\frac {1-d}{4} \geq a^2+ab+b^2>p$ when $(p, ab)>1$ and $a, b$ are both non-negative or non-positive.  When $ab<0$, we note that $a^2+ab+b^2=(a+b)^2-ab>p$ when $(p, ab)>1$. So in either case this we must have $(ab,p)=1$. \newline

  For any fixed prime $\varpi$, we define
\begin{align*}
  E_K := \sum_{x \bmod \varpi}\tilde{e}_K \left( \frac x \varpi \right).
\end{align*}
  Let $c \pmod \varpi$ be such that $\tilde{e}_K \left( \frac c \varpi \right) \neq 1$. We note that such $c$ must exist, for otherwise $g_K(\chi_{\varpi})=\sum_{x \pmod{q}} \chi(x)=0$, contradicting the well-known fact that $|g_K(\chi_{\varpi})|=N(\varpi)^{1/2}$. Then
\begin{align*}
  \tilde{e}_K \left( \frac c \varpi \right)E_K = \sum_{x \bmod \varpi}\tilde{e}_K \left( \frac {x+c} \varpi \right)=E_K.
\end{align*}
  The above then implies that $E_K=0$. Now the arguments given in the proof of \cite[Lemma 2.6]{G&Zhao16} carry through with the above observations to show that the assertion of the lemma is valid for any $P$-primary prime $\varpi$ in \cite[Lemma 2.6]{G&Zhao16}. \newline

  Now, note that we have the easily verified relation that for any odd $n \in \mathcal O_K$ and any $u \in U_K$,
\begin{align*}
   g_K(\chi_{un})=\leg {u}{n}g_K(\chi_{n}).
\end{align*}
 As $\leg {u}{n}=1$ for any $u \in U_K$ and any odd $n$ with $N(n) \equiv 1 \pmod 4$ by \eqref{2.051} and \eqref{supprulequadtic},  we see that in order to establish \eqref{gvalue} for the case $N(n)\equiv 1 \pmod 4$, it suffices to show that this is so for any generator of the ideal $(n)$. In particular,  the validity of \eqref{gprimevalue} implies that \eqref{gvalue} is true for any primary prime $\omega$ when $N(\varpi)\equiv 1 \pmod 4$. \newline

  Next, one checks via \eqref{primary} that when a primary prime $\varpi$ has norm congrudent to $-1$ modulo $4$, then it is also $P$-primary when $d \neq -2, -7$, while $-\varpi$ is $P$-primary when $d = -2, -7$. In fact, as $N(c)\equiv 1 \pmod 4$ for any $c \in G^2_4$, it suffices to check this for elements $c \in <1+2\omega_K>$ for $d \neq -2$ and for elements $c \in \{ 1, -(1+\omega_K)\}$ when $d=-2$, which in turn can be easily verified. It follows from this and \eqref{gprimevalue} that \eqref{gvalue} holds for all primary primes. \newline

  Lastly, to establish the general case, we note that it follows from the definition of $g_K$ that for primary $n_1, n_2$ with $(n_1, n_2)=1$, we have
\begin{align*}
   g_{K}(\chi_{n_1n_2}) =& \leg{n_2}{n_1}\leg{n_1}{n_2}g_{K}(\chi_{n_1}) g_{K}(\chi_{n_1}).
\end{align*}
  The general case of the assertion of the lemma now follows from the above, the case when $n$ is primary prime and Lemma \ref{Quadrecgenprim} by induction on the number of primary primes dividing $n$.
\end{proof}

 Recall that $\chi^{(c_Kc)}$ is a primitive quadratic character of trivial infinite type for any square-free $c$. The associated Hecke $L$-function satisfies then functional equation given in \eqref{fneqn} below. It is expected that the root number $W(\chi)$ there is $1$ for primitive quadratic Hecke characters. Our next lemma confirms this.
\begin{lemma}
\label{lem: primquadGausssum}
  Let $K=\mq(\sqrt{d})$ with $d \in \mathcal S$. For any odd, square-free $c \in \mathcal{O}_K$, we have
\begin{align} \label{primquadGausssum}
 g_K(\chi^{(c_Kc)})=\sqrt{N(c_Kc )}.
\end{align}
\end{lemma}
\begin{proof}
   Note first that \eqref{primquadGausssum} is established in \cite[Lemma 2.2]{Gao2} for $d=-1$.  So we may assume that $d \neq -1$ in the rest of the proof. It follows from the Chinese remainder theorem that $x = c_Ky +c z$ varies over the residue class modulo $c_Kc$ as $y$ and $z$ vary over the residue class modulo $c$ and $c_K$, respectively.  Thus
\begin{align*}
 g_K(\chi^{(c_Kc)})=& \sum_{z \bmod{c_K}}\sum_{y \bmod{c}} \leg{c_K}{c_Ky +c z}\leg{c}{c_Ky +c z} \widetilde{e}_K\leg{y}{c}\widetilde{e}_K\leg{z}{c_K}.
\end{align*}
   As $\chi^{(c_K)}$ is a Hecke character of trivial infinite type  modulo $c_K$, we get that
\begin{align}
\label{1+ireci}
  \leg{c_K}{c_Ky +c z}=\chi^{(c_K)}(c_Ky +c z)=\chi^{(c_K)}(c z).
\end{align}

   On the other hand, let $s(z)$ denote the unique element in $U_K$ such that $s(z)z$ is primary for any $(z,2)=1$. It follows from the quadratic reciprocity law \eqref{quadreciKprim} and the observation that $N(cz) \equiv N(c_Ky +c z) \pmod 4$ that
\begin{align}
\label{varpireci}
 \leg{c}{c_Ky +c z}=\leg {s(z)c_ky}{c}(-1)^{((N(c)-1)/2)((N(s(z)cz)-1)/2)}.
\end{align}

   We then conclude from \eqref{1+ireci} and \eqref{varpireci} that
\begin{align*}
\begin{split}
  g_K(\chi^{(c_Kc)})=& \sum_{z \bmod{c_K}}\sum_{y \bmod{c}} \leg{c_K}{c z}\leg {s(z)c_Ky}{c} (-1)^{((N(c)-1)/2)((N(s(z)cz)-1)/2)}\widetilde{e}_K\leg{y}{c}\widetilde{e}_K\leg{z}{c_K} \\
=& \sum_{z \bmod{c_K}}\leg{c_K}{z} \leg {s(z)}{c} (-1)^{((N(c)-1)/2)((N(s(z)cz)-1)/2)} \widetilde{e}_K\leg{z}{c_K}   \sum_{y \bmod{c}}\leg {y}{c} \widetilde{e}_K\leg{y}{c}.
\end{split}
\end{align*}

  Observe that when $N(c) \equiv 1 \pmod 4$, we have $\leg {s(z)c_Ky}{c}=1$ and that $((N(c)-1)/2)((N(s(z)cz)-1)/2) \equiv 0 \pmod 2$. While when $N(c) \equiv -1 \pmod 4$,
\begin{align*}
 \frac {N(c)-1}{2}\frac {N(s(z)cz)-1}{2} \equiv \frac {N(s(z)cz)-1}{2} \equiv \frac {N(c)-1}{2}+\frac {N(s(z))-1}{2}+\frac {N(z)-1}{2} \equiv 1+\frac {N(z)-1}{2} \pmod 2.
\end{align*}

   We conclude from the above and Lemma \ref{Gausssum} that
\begin{align}
\label{gexpressionsimplified}
\begin{split}
  g_K(\chi^{(c_Kc)})=& \begin{cases}
    \sqrt{N(c)}\displaystyle \sum_{z \bmod{c_K}}\leg{c_K}{z}  \widetilde{e}_K\leg{z}{c_K} \qquad & \text{for all} \; d, \; N(c) \equiv 1 \pmod 4, \\
    i \sqrt{N(c)} \displaystyle \sum_{z \bmod{c_K}}\leg{c_K}{z} \leg {s(z)}{c}(-1)^{(N(z)-1)/2} \widetilde{e}_K\leg{z}{c_K} \qquad & \text{for} \; d \neq -2,-7, \; N(c) \equiv -1 \pmod 4,\\
     -i \sqrt{N(c)}\displaystyle \sum_{z \bmod{c_K}}\leg{c_K}{z} \leg {s(z)}{c}(-1)^{(N(z)-1)/2} \widetilde{e}_K\leg{z}{c_K} \qquad & \text{for} \; d = -2, -7, \; N(c) \equiv -1 \pmod 4.
\end{cases}
\end{split}
\end{align}

  Recall that $c_K=8$ if $d \neq -2$. Similar to \eqref{G4G2}, we see that as sets,
\begin{align*}
\begin{split}
   G_8 = G_4 \times \{ 1, 5, 1+4\omega_K, 5+4\omega_K \} \pmod 8,
\end{split}
\end{align*}
  We further note that for $a, b \in \mz$,
\begin{align*}
  \widetilde{e}_K\leg{a+b\omega_K}{c_K}=e \leg {b}{c_K}.
\end{align*}

  Note also the following supplementary law (see \cite[Propostion 4.2(iii)]{Lemmermeyer}) that asserts for $(m, 2)=1, m \in \mathcal O_K$,
\begin{align}
\label{2.05}
   \leg {2}{m}=\leg {2}{N(m)}_{\mz}.
\end{align}

  As we have $N(m) \equiv N(5m) \pmod 8$ for every $m \in G_4$, we deduce from \eqref{2.05} that $\leg {c_K}{m}=\leg {c_K}{5m}$ for these $m$. On the other hand, we have $e(b/8)=-e(5b/8)$ for any odd $b \in \mz$.  It follows that
\begin{align*}
\begin{split}
  \sum_{\substack{z=m, 5m, m \in G_4 \\ m=a+b\omega_K, a, b \in \mz}}\leg{c_K}{z}\widetilde{e}_K\leg{z}{c_K}=\sum_{\substack{z=m, 5m, m \in G_4 \\ m=a+b\omega_K, a, b \in \mz \\ 2|b}}\leg{c_K}{z}\widetilde{e}_K\leg{z}{c_K}.
\end{split}
\end{align*}
  We observe from \eqref{G2square} and \eqref{G4} that there are only four elements in $G_4$ that can be written as $a+b\omega_K$ with $a, b \in \mz$, $2|b$. These elements are precisely those in the set given in \eqref{classonemod4}. Note that the norms of these elements are all $\equiv \pm 1 \pmod 8$.  
  Using again $\leg{c_K}{m}=\leg{c_K}{5m}$ for any $m \in G_4$, we see that
\begin{align*}
\begin{split}
  \sum_{\substack{z=m, 5m, m \in G_4 \\ m=a+b\omega_K, a, b \in \mz \\ 2|b}}\leg{c_K}{z}\widetilde{e}_K\leg{z}{c_K}=2\sum_{\substack{z=m \in G_4 \\ m=a+b\omega_K, a, b \in \mz \\ 2|b}}\leg{c_K}{z}\widetilde{e}_K\leg{z}{c_K}.
\end{split}
\end{align*}
  Using \eqref{2.05}, a direct computation leads to
\begin{align}
\label{sum15result}
\begin{split}
  \sum_{\substack{z=m, 5m, m \in G_4 \\ m=a+b\omega_K, a, b \in \mz \\ 2|b}}\leg{c_K}{z}\widetilde{e}_K\leg{z}{c_K}=4.
\end{split}
\end{align}

  Similarly, we have $N((1+4\omega_K)m) \equiv N((5+4\omega_K)m) \pmod 8$ for every $m \in G_4$, so that $\leg {c_K}{(1+4\omega_K)m}=\leg {c_K}{(5+4\omega_K)m}$ for these $c$. Moreover, if we write $m=a+b\omega_K, (1+4\omega_K)m=a'+b'\omega_K$ with $a, b, a', b' \in \mz$, then it is easy to see that $b' \equiv b \pmod 4$.  It follows that
\begin{align*}
\begin{split}
  \sum_{\substack{z=(1+4\omega_K)m, (5+4\omega_K)m, m \in G_4 \\ m=a+b\omega_K, a, b \in \mz}}\leg{c_K}{z}\widetilde{e}_K\leg{z}{c_K}=& \sum_{\substack{z=(1+4\omega_K)m, (5+4\omega_K)m, m \in G_4 \\ m=a+b\omega_K, a, b \in \mz\\  2|b}}\leg{c_K}{z}\widetilde{e}_K\leg{z}{c_K} \\
=& 2\sum_{\substack{z=(1+4\omega_K)m, m \in G_4 \\ m=a+b\omega_K, a, b \in \mz\\  2|b}}\leg{c_K}{z}\widetilde{e}_K\leg{z}{c_K}.
\end{split}
\end{align*}

A direct computation renders
\begin{align}
\label{sum15-1result}
\begin{split}
  \sum_{\substack{z=(1+4\omega_K)m, (5+4\omega_K)m, m \in G_4 \\ m=a+b\omega_K, a, b \in \mz}}\leg{c_K}{z}\widetilde{e}_K\leg{z}{c_K}=4.
\end{split}
\end{align}

  Our discussions above apply similarly to the sum
\begin{align*}
\begin{split}
\sum_{z \bmod{c_K}}\leg{c_K}{z} \leg {s(z)}{c}(-1)^{(N(z)-1)/2} \widetilde{e}_K\leg{z}{c_K}.
\end{split}
\end{align*}

As $s(z)=s(5z)$ and that $(N(5z)-1)/2 \equiv (N(5)-1)/2+(N(z)-1)/2 \equiv (N(z)-1)/2 \pmod 2$, we see that
\begin{align}
\label{sum15twist}
\begin{split}
  \sum_{\substack{z=m, 5m, m \in G_4 \\ m=a+b\omega_K, a, b \in \mz}}\leg{c_K}{z}\leg {s(z)}{c}(-1)^{(N(z)-1)/2}\widetilde{e}_K\leg{z}{c_K}=& 2\sum_{\substack{z=m, m \in G_4 \\ m=a+b\omega_K, a, b \in \mz \\ 2|b}}\leg{c_K}{z}\leg {s(z)}{c}(-1)^{(N(z)-1)/2}\widetilde{e}_K\leg{z}{c_K} \\
=& \begin{cases}
-4i, & d \neq -7, \\
4i, & d = -7.
\end{cases}
\end{split}
\end{align}
 Here we have $N(1+2\omega_K)  \equiv-1 \pmod 8$ when $d \neq -7$ and  $N(1+2\omega_K)  \equiv 3 \pmod 8$ when $d = -7$.

 Similarly, we have $s((1+4\omega_K)z) = s((5+4\omega_K)z)$ and that $(N((1+4\omega_K)z)-1)/2 \equiv (N(1+4\omega_K)-1)/2+(N(z)-1)/2 \equiv (N(z)-1)/2 \equiv (N((5+4\omega_K)z)-1)/2 \pmod 2$, so that we have
\begin{align}
\label{sum15-1twist}
\begin{split}
  \sum_{\substack{z=(1+4\omega_K)m, (5+4\omega_K)m, m \in G_4 \\ m=a+b\omega_K, a, b \in \mz}} & \leg{c_K}{z}\leg {s(z)}{c}(-1)^{(N(z)-1)/2}\widetilde{e}_K\leg{z}{c_K} \\
=& 2\sum_{\substack{z=(1+4\omega_K)m, m \in G_4 \\ m=a+b\omega_K, a, b \in \mz\\  2|b}}\leg{c_K}{z}\leg {s(z)}{c}(-1)^{(N(z)-1)/2}\widetilde{e}_K\leg{z}{c_K} = \begin{cases}
-4i, & d \neq -7, \\
4i, & d = -7.
\end{cases}
\end{split}
\end{align}
Here $N(1+4\omega_K)  \equiv 5 \pmod 8$. We conclude from \eqref{gexpressionsimplified}, \eqref{sum15result}, \eqref{sum15-1result}, \eqref{sum15twist} and \eqref{sum15-1twist} that the assertion of the lemma holds for $d \neq -2$. \newline

  Lastly, for the case $d=-2$, we have $c_K=4\omega_K$ in this case. Observe that the residue classes modulo $(\omega_K)$ consists of $\{0, 1\}$ and that $1+4\{0,1\}$ are two distinct elements modulo $4\omega_K$ but both congruent to $1$ modulo $4$. Thus, we have similar to \eqref{G4G2} that as sets,
\begin{align*}
\begin{split}
   G_{c_K} = G_4 \times \{ 1, 5 \} \pmod {c_K}.
\end{split}
\end{align*}
  Note that this time we have $\widetilde{e}_K ( m/c_K) =e(-a/8)$ for any $m=a+b\omega_K$ with $a$, $b \in \mz$. Moreover, we have by \cite[Propostion 4.2(iii)]{Lemmermeyer},
\begin{align*}
\begin{split}
  & \leg {\omega_K}{1+\omega_K}= \leg {\omega_k-(1+\omega_K)}{1+\omega_K}=\leg {-1}{1+\omega_K}=(-1)^{(N(1+\omega_K)-1)/2}=-1, \\
  & \leg {\omega_K}{5}=\leg {N(\omega_K)}{5}_{\mz}=\leg {2}{5}_{\mz}=-1.
\end{split}
\end{align*}

  One checks directly that the sums given in \eqref{sum15result} and \eqref{sum15twist} this time  equal $4\sqrt{2}$, $4\sqrt{2}i$, respectively.  We conclude from this and \eqref{gexpressionsimplified} that the assertion of the lemma holds for $d = -2$ as well. This completes the proof of the lemma.
\end{proof}

\subsection{Functional equations for Hecke $L$-functions}
	
Let $K=\mq(\sqrt{d})$ with $d \in \mathcal S$ and let $\chi$ be a primitive quadratic Hecke character of trivial infinite type modulo $q$. A well-known result of E. Hecke asserts that $L(s, \chi)$ has an analytic continuation to the whole complex plane and satisfies the functional equation (see \cite[Theorem 3.8]{iwakow})
\begin{align}
\label{fneqn}
  \Lambda(s, \chi) = W(\chi)\Lambda(1-s, \chi), \; \mbox{where} \;  W(\chi) = g_K(\chi)(N(q))^{-1/2}
\end{align}
and
\begin{align}
\label{Lambda}
  \Lambda(s, \chi) = (|D_K|N(q))^{s/2}(2\pi)^{-s}\Gamma(s)L(s, \chi).
\end{align}

  We apply the above to the case when $\chi=\chi^{(c_Kc)}$ for odd, square-free $c \in \mathcal O_K$ and deduce from Lemma \ref{lem: primquadGausssum} that $W(\chi^{(c_Kc)})=1$ in our situation.  It then follows from \eqref{fneqn} and \eqref{Lambda} that
\begin{align}
\label{fneqnL}
  L(s, \chi^{(c_Kc)})=(|D_K|N(c_Kc))^{1/2-s}(2\pi)^{2s-1}\frac {\Gamma(1-s)}{\Gamma (s)}L(1-s, \chi^{(c_Kc)}).
\end{align}

\subsection{A mean value estimate for quadratic Hecke $L$-functions}
 In the proof of Theorem \ref{Theorem for all characters}, we need the following lemma, which gives an upper bound for the second moment of quadratic Hecke $L$-functions.
\begin{lemma}
\label{lem:2.3}
 With the notation as above, let $S(X)$ denote the set of Hecke characters $\chi^{(m)}$ with $N(m)$ not exceeding $X$. Then we have, for any complex number $s$  and any $\varepsilon>0$ with $\Re(s) \geq 1/2$, $|s-1|>\varepsilon$,
\begin{align}
\label{L4est}
\sum_{\substack{\chi \in S(X)}} |L(s, \chi)| \ll & X^{1+\varepsilon}|s|^{1/2+\varepsilon}.
\end{align}
\end{lemma}
\begin{proof}
Let $S'(X)$ stand for the set of $\chi^{(m)}$ with $m$ being square-free and $N(m) \leq X$.  It follows from the arguments in the proof of \cite[Corollary 1.4]{BGL}, upon using the quadratic large sieve for number fields \cite[Theorem 1.1]{G&L} that we have
\begin{align*}
\sum_{\substack{\chi \in S'(X)}}|L(s, \chi)|^2
\ll & X^{1+\varepsilon} |s|^{1+\varepsilon}.
\end{align*}
  It follows from the above and the Cauchy-Schwarz inequality that the estimation given in \eqref{L4est} is valid with $S(X)$ replaced by $S'(X)$. Now, we write $m=m_1m^2_2$ with $m_1$ being square-free to see that
\begin{align*}
\sum_{\substack{\chi \in S(X)}} |L(s, \chi)| \ll \sum_{N(m_2)\geq 1}\sum_{\substack{\chi \in S'(X/N(m^2_2))}}|L(s, \chi)| \ll & X^{1+\varepsilon} |s|^{1/2+\varepsilon}.
\end{align*}
  This completes the proof of the lemma.
\end{proof}

\subsection{Some results on multivariable complex functions}
	
   We include in this section some results from multivariable complex analysis. First we need the concept of a tube domain.
\begin{defin}
		An open set $T\subset\mc^n$ is a tube if there is an open set $U\subset\mr^n$ such that $T=\{z\in\mc^n:\ \Re(z)\in U\}.$
\end{defin}
	
   For a set $U\subset\mr^n$, we define $T(U)=U+i\mr^n\subset \mc^n$.  We have the following Bochner's Tube Theorem \cite{Boc}.
\begin{theorem}
\label{Bochner}
		Let $U\subset\mr^n$ be a connected open set and $f(z)$ be a function holomorphic on $T(U)$. Then $f(z)$ has a holomorphic continuation to the convex hull of $T(U)$.
\end{theorem}

The convex hull of an open set $T\subset\mc^n$ is denoted by $\widehat T$.  Then we quote the result from \cite[Proposition C.5]{Cech1} on the modulus of holomorphic continuations of functions in multiple variables.
\begin{prop}
\label{Extending inequalities}
		Assume that $T\subset \mc^n$ is a tube domain, $g,h:T\rightarrow \mc$ are holomorphic functions, and let $\tilde g,\tilde h$ be their holomorphic continuations to $\widehat T$. If  $|g(z)|\leq |h(z)|$ for all $z\in T$, and $h(z)$ is nonzero in $T$, then also $|\tilde g(z)|\leq |\tilde h(z)|$ for all $z\in \widehat T$.
\end{prop}

\section{Proof of Theorem \ref{Theorem for all characters}}
\label{sec: proof of main result}

The Mellin inversion renders
\begin{align}
\label{Integral for all characters}
		\sumstar_{(c,2)=1}  \frac{L(\frac{1}{2}+\alpha,\chi^{(c_Kc)})}{L(\frac{1}{2}+\beta,\chi^{(c_Kc)})} w\left( \frac {N(c)}X \right) =\frac1{2\pi i}\int\limits_{(2)}A_K\lz s,\tfrac12+\alpha,\tfrac12+\beta\pz X^s\widehat w(s) \dif s,
\end{align}
  where
\begin{align*}
\begin{split}
A_K(s,w,z) := \sumstar_{\substack{c \odd }}\frac{L(w, \chi^{(c_kc)})}{L(z,\chi^{(c_Kc)})N(c)^s},
\end{split}
\end{align*}
 and $\widehat{w}$ is the Mellin transform of $w$ defined by
\begin{align*}
     \widehat{w}(s) =\int\limits^{\infty}_0w(t)t^s\frac {\dif t}{t}.
\end{align*}

  It follows from \eqref{Integral for all characters} that in order to prove Theorem \ref{Theorem for all characters}, it suffices to understand the analytical properties of $A_K(s,w,z)$. To that end, we write $\mu_K$ for the M\"obius function on $\mathcal O_K$.  For $\Re(s), \Re(w), \Re(z)$ large enough,
\begin{align}
\label{Aswzexp}
\begin{split}
A_K(s,w,z)
=\sumstar_{\substack{c \odd }}\sum_{\substack{m,k}}\frac{\mu_K(k)\chi^{(c_Kc)}(k)\chi^{(c_Kc)}(m)}{N(k)^zN(m)^wN(c)^s} = \sum_{\substack{m, k}}\frac{\mu_K(k)\chi^{(c_K)}(mk)}{N(m)^wN(k)^z}\sumstar_{\substack{c \odd}}\frac{\chi_{mk}(c)}{N(c)^s}.
\end{split}
\end{align}

 Moreover, we apply the functional equation \eqref{fneqnL} to $L(w, \chi^{(c_Kc)})$ and derive from \eqref{Aswzexp} that
\begin{align}
\label{Aswzexpzfcneq}
\begin{split}
 A_K(s,w,z)=& (|D_K|N(c_K))^{1/2-w}(2\pi)^{2w-1}\frac {\Gamma(1-w)}{\Gamma (w)}\sumstar_{\substack{c \odd}}\frac{L(1-w, \chi^{(c_Kc)})}{L(z, \chi^{(c_Kc)})N(c)^{s+w-1/2}} \\
=& (|D_K|N(c_K))^{1/2-w}(2\pi)^{2w-1}\frac {\Gamma(1-w)}{\Gamma (w)}A_K(s+w-\tfrac 12,1-w,z).
\end{split}
\end{align}
  The above can be regarded as a functional equation for $A_K(s,w,z)$.
	
\subsection{Regions of absolute convergence of $A_K(s,w,z)$}

  Note that when $c$ is primary, we have by Lemma \ref{Quadrecgenprim} that for odd $m, k$,
\begin{align*}
\begin{split}
 \chi_{mk}(c)=\chi^{s(mk)(mk)}(c)(-1)^{(N(mk)-1)/2\cdot (N(c)-1)/2}=\leg {16s(mk)(mk)(-1)^{(N(mk)-1)/2}}{c} :=\chi^{*}_{mk}(c),
\end{split}
\end{align*}
  where we recall that $s(z)$ denotes the unique element in $U_K$ such that $s(z)z$ is primary for any $(z,2)=1$ and the second equality above follows from
\eqref{2.051}. Here $\chi^{*}_{mk}$ is a real Hecke character of trivial infinite type whose conductor divides $16mk$. \newline

The presence of $\chi^{(c_K)}(mk)$ in \eqref{Aswzexp} restricts the sums there to over odd $m$, $k$. This allows us to write the inner sum in the last expression of \eqref{Aswzexp} as a Euler product, getting
\begin{align}
\label{sumoverd}
\begin{split}
\sumstar_{\substack{c \odd}}\frac{\chi_{mk}(c)}{N(c)^s}=\prod_{(\varpi,2)=1}(1+\chi^*_{mk}(\varpi)N(\varpi)^{-s})=\frac {L(s, \chi^*_{mk})}{\zeta^{(2mk)}_K(2s)}=\frac {L(s, \chi^*_{mk})}{\zeta_K(2s)}\prod_{\varpi|2mk}(1-N(\varpi)^{-2s})^{-1}.
\end{split}
\end{align}

  Thus \eqref{Aswzexp} and \eqref{sumoverd} yield
\begin{align}
\label{Aswd}
\begin{split}
A_K(s,w,z)=\sum_{\substack{m, k}}\frac{\mu_K(k)\chi^{(c_K)}(mk)}{N(m)^wN(k)^z}\frac {L(s, \chi^*_{mk})}{\zeta_K(2s)}\prod_{\varpi|2mk}(1-N(\varpi)^{-2s})^{-1}.
\end{split}
\end{align}

  Observe that when $\Re(s)>0$,
\begin{align*}
 \prod_{\varpi|2mk}(1-N(\varpi)^{-2s})^{-1} \leq \prod_{\varpi|2mk} \Big |(1-2^{-2\Re(s)})^{-1} \Big | \leq ((1-2^{-2\Re(s)})^{-1})^{\mathcal{W}_K(2mk)}
\ll N(mk)^{\varepsilon},
\end{align*}
  where $\mathcal{W}_K$ denotes the number of distinct prime factors of $n$ and the last estimation above follows from the well-known bound
(which can be derived in a manner similar to the proof of the classical case over $\mq$ given in\cite[Theorem 2.10]{MVa1})
\begin{align*}
   \mathcal{W}_K(h) \ll \frac {\log N(h)}{\log \log N(h)}, \quad \mbox{for} \quad N(h) \geq 3.
\end{align*}

From \eqref{Aswd} and the above, we infer that for $\Re(s)>0$,
\begin{align}
\label{Aswboundsumm}
\begin{split}
 \zeta_K(2s)A_K(s,w,z) \ll &  \sum_{\substack{m, k}}\frac{|L(s, \chi^*_{mk})|}{N(m)^{\Re (w)-\varepsilon}N(k)^{\Re(z)-\varepsilon}}.
\end{split}
\end{align}

Now by Lemma \ref{lem:2.3},
\begin{align}
\label{sumovermk}
\begin{split}
  \sum_{N(mk) \leq X} |L(s, \chi^*_{mk})| \ll \sum_{N(c) \leq X} d_K(c)|L(s, \chi^*_{c})| \ll X^{\varepsilon}\sum_{N(c) \leq X} |L(s, \chi^*_{c})| \ll X^{1+\varepsilon}|s|^{1/2+\varepsilon},
\end{split}
\end{align}
   where $d_K(n)$ is the divisor function on $\mathcal O_K$ and we have, similar to the classical bound for the divisor function on $\mz$ given in \cite[Theorem 2.11]{MVa1}, that
\begin{align*}
\begin{split}
 d_K(n) \ll N(n)^{\varepsilon}.
\end{split}
\end{align*}

   We thus conclude from \eqref{Aswboundsumm}, \eqref{sumovermk} and partial summation that the function $A_K(s,w,z)$ converges absolutely in the region
\begin{equation*}
		S_{1,1}=\{(s,w,z): \Re(s)>0, \ \Re(w)>1,\ \Re(z)>1, \ \Re(s+w)>3/2,\ \Re(s+z)>3/2\}.
\end{equation*}

  In what follows, we shall define similar regions $S_{i, j}$ and $S_j$ and adopt the convention that for any real number $\delta$,
\begin{align*}
\begin{split}
 S_{i, j,\delta} := \{ (s,w)+\delta (1,1) : (s,w) \in S_{i,j} \} \quad \mbox{and} \quad S_{j,\delta} :=  \{ (s,w)+\delta (1,1) : (s,w) \in S_j \}.
\end{split}
\end{align*}

   Using this notation, we see that in the region $S_{1,1, \varepsilon}$,
\begin{align}
\label{Aswboundwlarge}
\begin{split}
 |(s-1)\zeta_K(2s)A_K(s,w,z)| \ll & (1+|s|)^{3/2+\varepsilon}.
\end{split}
\end{align}

  We note also from \eqref{Aswzexp} that
\begin{align}
\label{Aswbound}
\begin{split}
A_K(s,w,z) \ll & \sumstar_{c \odd}\frac{|L(w, \chi^{(c_Kc)})|}{|L(z,\chi^{(c_Kc)})|N(c)^{\Re (s)}}.
\end{split}
\end{align}
 When $\Re(w) \geq 1/2$, we apply from \eqref{L4est} and partial summation to infer that the sum above is convergent in the region
\begin{equation*}
		S_{1,2}=\{(s,w): \ \Re(s)>1, \ \Re(w) \geq 1/2, \ \Re(z)>1/2 \}.
\end{equation*}
Recall from \cite[Theorem 5.19]{iwakow} that, on GRH,
\begin{align}
\label{Lindelof}
\begin{split}
      \frac1{|L(z,\chi^{(c_Kc)})|}\ll & (|z|N(c_Kc))^{\varepsilon}, \quad \Re(z) \geq 1/2+\varepsilon.
\end{split}
\end{align}
Thus in the region $S_{1,2, \varepsilon}$, 
\begin{align}
\label{Aswboundslarge}
\begin{split}
 |A_K(s,w, z)| \ll &  |z|^{\varepsilon}(1+|w|)^{1/2+\varepsilon}.
\end{split}
\end{align}

If $\Re(w) < 1/2$, we apply first the functional equation \eqref{fneqnL} and then deduce from \eqref{L4est} via partial summation that the sum in \eqref{Aswbound} is also convergent in the region
\begin{equation*}
		S_{1,3}=\{(s,w): \ \Re(s)>1, \ \Re(w)<1/2, \ \Re(s+w) > 3/2, \ \Re(z)>1/2 \}.
\end{equation*}
 now Stirling's formula (see \cite[(5.113)]{iwakow}) implies that
\begin{align}
\label{Stirlingratio}
  \frac {\Gamma(1-s)}{\Gamma (s)} \ll (1+|s|)^{1-2\Re (s)}.
\end{align}
This estimate, together with our discussions above, implies that in the region $S_{1,3, \varepsilon}$, 
\begin{align}
\label{Aswboundslarge1}
\begin{split}
 |A_K(s,w, z)| \ll &  |z|^{\varepsilon}(1+|w|)^{1/2+1-2\Re(w)+\varepsilon}.
\end{split}
\end{align}

  We denote $S_{1,4}$ for the union of $S_{1,2}$ and $S_{1,3}$ so that
\begin{equation*}
		S_{1,4}=\{(s,w): \ \Re(s)>1,  \ \Re(s+w) > 3/2, \ \Re(z)>1/2 \}.
\end{equation*}
  We then deduce from \eqref{Aswboundslarge} and \eqref{Aswboundslarge1} that in the region $S_{1,4, \varepsilon}$, under GRH,
\begin{align}
\label{Aswboundslarge2}
\begin{split}
 |A_K(s,w, z)| \ll &  |z|^{\varepsilon}(1+|w|)^{1/2+\max \{1-2\Re(w), 0\} +\varepsilon}.
\end{split}
\end{align}

The point $(1/2, 1, 1)$ is in $S_{1,1}$ while $(1,1/2, 1/2)$, $(1, 1/2, 1)$ are in $S_{1,2}$ and the three points determine a plane whose equation is given by $\Re(s+w)=3/2$. Similarly, the point $(1/2, 1, 1)$ is in $S_{1,1}$ and $(1,1/2, 1/2)$, $(1, 1, 1/2)$ are in $S_{1,2})$ and the three points determine a plane whose equation is $\Re(s+z)=3/2$. It follows that the convex hull of $S_{1,1}$ and $S_{1,4}$ equals
\begin{equation*}
S_1=\{(s,w,z): \ \Re(s)> 0, \ \Re(z)> 1/2, \ \Re(s+w)> 3/2,\ \Re(s+z)> 3/2 \}.
\end{equation*}

Moreover, Proposition \ref{Extending inequalities}, \eqref{Aswboundwlarge}, \eqref{Aswboundslarge2} and the bound $\zeta_K(2s) \ll 1$ for $\Re(s)>1$ yeild that in the region $S_{1, \varepsilon}$, on GRH,
\begin{align}
\label{Aswboundwlarge3}
\begin{split}
 |(s-1)\zeta_K(2s)A_K(s,w,z)| \ll & (1+|s|)^{3/2+\varepsilon}|z|^{\varepsilon}(1+|w|)^{1/2+\max \{1-2\Re(w), 0\}+\varepsilon}.
\end{split}
\end{align}

  We now apply the functional equation for $A_K(s,w,z)$ given in \eqref{Aswzexpzfcneq} and follow our discussions above to see that $A_K(s,w,z)$ is also  holomorphic in the region
\begin{equation*}
		S_2=\{(s,w,z): \ \Re(s+w)> 1/2, \ \Re(z)> 1/2, \ \Re(s)>1,\ \Re(s+w+z)> 2\}.
\end{equation*}
  We also deduce from \eqref{Aswzexpzfcneq}, \eqref{Stirlingratio} and \eqref{Aswboundwlarge3} that in the region $S_{2, \varepsilon}$ for any $\varepsilon>0$, we have on GRH,
\begin{align}
\label{Aswboundwlarge4}
\begin{split}
 |(s+w-3/2) & \zeta_K(2s+2w-1)A_K(s,w,z)| \\
 \ll  & (1+|s+w-1/2|)^{3/2+\varepsilon}|z|^{\varepsilon}(1+|1-w|)^{1/2+\max \{1-2\Re(w), 0\}+\varepsilon}(1+|w|)^{1-2\Re (w)} \\
 \ll & (1+|s|)^{3/2+\varepsilon}|z|^{\varepsilon}(1+|w|)^{3-2\Re (w)+\max \{1-2\Re(w), 0\}+\varepsilon}.
\end{split}
\end{align}

  The union of $S_1$ and $S_2$ is connected and the convex hull of $S_1, S_2$ equals
\begin{align}
\begin{split}
	&	S_3 =\{(s,w, z): \ \Re(s)>0, \ \Re(z)> 1/2, \ \Re(s+w+z)> 2, \\
& \hspace{0.5in} \ \Re(2s+w+z)>3, \ \Re(2s+w)> 3/2, \ \Re (s+w)>1/2, \ \Re(s+z)> 3/2 \}.
\end{split}
\end{align}

  To see this, note that the two planes $\Re(s+w)= 3/2$ and $\Re(s+z)= 3/2$ intersect when $z=1/2$ at the line given by the intersection of $\Re(s+w)= 3/2$ and $\Re(s)= 1$. Note further that the plane $\Re(s+w+z)= 2$ can be written as $\Re(s+w)= 2-z$ and $2-z=3/2$ for $z=1/2$. Note also that $2-z=1/2$ when $z=3/2$. It follows that the convex hull of $S_1$ and $S_2$ contains the plane determined by the two lines:
$\Re(s+w)= 3/2$, $\Re(s+z)=3/2$ on $S_1$ and $\ \Re(s+w+z)=2,\ \Re(s)=1$ on $S_2$, intersecting with the two planes: $z=1/2$ and $z=3/2$. Note that the two lines both intersect the plane $z=1/2$ at the point $(1, 1/2, 1/2)$. The first line intersects the plane $z=3/2$ at $(0, 3/2, 3/2)$, the second at $(1, -1/2, 3/2)$. These three points then determine a plane with the equation $\Re(2s+w+z)=3$.  Moreover, when $z>3/2$, the convex hull continues with the plane that is determined by the lines: $\Re(s)=0, \ \Re(s+w)=3/2$ in $S_1$ and $\Re(s)=1, \ \Re(s+w)=1/2$ in $S_2$. As the point $(0, 3/2, 3/2)$ is on the first line and the point $(1, -1/2, 3/2)$ is on the second line, we see easily that the plane has the equation: $\Re(2s+w)= 3/2$.

Consequently, Theorem \ref{Bochner} gives that $(s-1)(s+w-1)\zeta_K(2s)\zeta_K(2w+2s-1)A_K(s,w,z)$ converges absolutely in the region $S_3$. Moreover, Proposition~\ref{Extending inequalities} supplies us with the bound, inherited from \eqref{Aswboundwlarge3} and \eqref{Aswboundwlarge4}, that in the region
\begin{equation*}
		S_{\varepsilon} := S_{3, \varepsilon} \cap \{(s,w,z):\ \Re(s) \geq 1/2+\varepsilon, \ \Re(s+w) \geq 1+\varepsilon, \ \Re(z) \geq 1/2+\varepsilon \},
\end{equation*}
for any $0<\varepsilon<1/100$,
\begin{align*}
\begin{split}
 |(s-1)(s+w-1)\zeta_K(2s)\zeta_K(2w+2s-1)A(s,w)| \ll & |\zeta_K(2s)\zeta_K(2w+2s-1)|(1+|s|)^{5/2+\varepsilon}|z|^{\varepsilon}(1+|w|)^{5/2+\varepsilon}.
\end{split}
\end{align*}

Furthermore, when $\Re(s)  \geq 1/2+\varepsilon$ and $\Re(s+w) \geq 1+\varepsilon$, we have $1 \ll \zeta(2s)$ and $\zeta(2w+2s-1) \ll 1$.  So from \eqref{Aswboundslarge1}, in the region $S_{\varepsilon}$,
\begin{align}
\label{Aswboundslarge6}
\begin{split}
 |(s-1)(s+w-1)A_K(s,w,z)| \ll & (1+|s|)^{5/2+\varepsilon}|z|^{\varepsilon}(1+|w|)^{5/2+\varepsilon}.
\end{split}
\end{align}

\subsection{Residues of $A_K(s,w,z)$ at $s=1$ and $s+w=3/2$}
\label{sec:resA}

It follows from \eqref{sumoverd} that $A_K(s,w,z)$ has a pole at $s=1$ arising from the terms with $mk=u\square$ with $u \in U_K$, where $\square$ denotes a perfect square. In this case,
\begin{equation*}
		L\lz s, \chi^*_{mk}\pz=\zeta_K(s)\prod_{\varpi|2mk}\lz1-\frac1{N(\varpi)^s}\pz.
\end{equation*}
    Recall that we denote the residue of $\zeta_K(s)$ at $s = 1$ by $r_K$ and for any $n \in \mathcal O_K$. We then deduce from \eqref{Aswd} that
\begin{align}
\label{ResA0}
 \res_{s=1}A_K(s,w,z)=\frac {r_K |U_K|^2}{\zeta_K(2)} \sum_{\substack{mk=\square \\m,k\odd}}\frac{\mu_{K}(k)a(2mk)}{N(m)^wN(k)^z}=
\frac {r_K |U_K|^2a(2)}{\zeta_K(2)}  \sum_{\substack{mk=\square \\m,k\odd}}\frac{\mu_{K}(k)a(mk)}{N(m)^wN(k)^z},
\end{align}
where $a(n)$ is defined in \eqref{adef}. \newline

  We now express the last sum in \eqref{ResA0} as an Euler product via a direct computation similar to that done in \cite[(4.11)]{Cech1} to obtain that
\begin{align}
\label{ResA}
 \res_{s=1}A_K(s,w,z)=
\frac {r_K |U_K|^2a(2)}{\zeta_K(2)}  \frac{\zeta_K^{(2)}(2w)}{\zeta_K^{(2)}(w+z)}P(w,z),
\end{align}
  where $P(w,z)$ is given in \eqref{Pwz}. \newline

  Similarly, we deduce from \eqref{Aswzexpzfcneq} that
\begin{align}
\label{ResA1}
\begin{split}
 \res_{s=3/2-w}A_K(s,w,z)
=& (|D_K|N(c_K))^{1/2-w}(2\pi)^{2w-1}\frac {\Gamma(1-w)}{\Gamma (w)}\frac {r_K |U_K|^2a(2)}{\zeta_K(2)}  \frac{\zeta_K^{(2)}(2-2w)}{\zeta_K^{(2)}(1-w+z)}P(1-w,z).
\end{split}
\end{align}
	
\subsection{Completion of proof}

Now repeated integration by parts gives that for any integer $A \geq 0$,
\begin{align}
\label{whatbound}
 \widehat w(s)  \ll  \frac{1}{(1+|s|)^{A}}.
\end{align}

We evaluate the integral in \eqref{Integral for all characters} by shifting the line of integration  to $\Re(s)=E(\alpha,\beta)+\varepsilon$,  where $E(\alpha,\beta)$ is given in \eqref{Nab}.  Applying \eqref{Aswboundslarge6} and \eqref{whatbound} gives that the integral on the new line can be absorbed into the $O$-term in \eqref{Asymptotic for ratios of all characters}.  We encounter simple poles at $s=1$ and $s=1-\alpha$ in the process whose residues are given in \eqref{ResA} and \eqref{ResA1}, respectively. This yields the main terms in \eqref{Asymptotic for ratios of all characters} and completes the proof of Theorem \ref{Theorem for all characters}.

\section{Proof of Theorem \ref{Theorem for log derivatives}}

We first recast the expression in \eqref{Asymptotic for ratios of all characters} as
\begin{align}
\label{Asymptoticshortversion}
\begin{split}	
&\sumstar_{(c,2)=1}  \frac{L(\frac{1}{2}+\alpha,\chi^{(c_Kc)})}{L(\frac{1}{2}+\beta,\chi^{(c_Kc)})} w\left( \frac {N(c)}X \right) =   XM_1(\alpha,\beta)+X^{1-\alpha}M_2(\alpha,\beta)+R(X,\alpha,\beta),
\end{split}
\end{align}
   where $R(X,\alpha,\beta)$ is the error term in \eqref{Asymptotic for ratios of all characters} and we define
\begin{align*}
\begin{split}
	M_1(\alpha,\beta)=& X\M w(1)\frac {r_K |U_K|^2a(2)}{\zeta_K(2)} \frac{\zeta_K^{(2)}(1+2\alpha)}{\zeta_K^{(2)}(1+\alpha+\beta)}
P(\tfrac 12+\alpha, \tfrac 12+\beta), \\
M_2(\alpha,\beta) =& X^{1-\alpha}\M w(1-\alpha)(|D_K|N(c_K))^{-\alpha}(2\pi)^{2\alpha}\frac {\Gamma(\tfrac{1}{2}-\alpha)}{\Gamma (\tfrac{1}{2}+\alpha)}\frac {r_K |U_K|^2a(2)}{\zeta_K(2)}  \frac{\zeta_K^{(2)}(1-2\alpha)}{\zeta_K^{(2)}(1-\alpha+\beta)}P(\tfrac 12-\alpha,\tfrac 12+\beta).
\end{split}
\end{align*}
	Note that the expression on the left-hand side of \eqref{Asymptoticshortversion} and both $M_1(\alpha,\beta)$ and $M_2(\alpha,\beta)$ are analytic functions of $\alpha,\beta$, so is $E(X,\alpha,\beta)$. \newline
	
	We now differentiate the above terms with respect to $\alpha$ for a fixed $\beta=r$ with $\Re(\beta)>\varepsilon$, and then set $\alpha=\beta=r$ to see that
\begin{align}
\label{M1der}
\begin{split}
			\frac{\dif}{\dif \alpha} XM_1(\alpha,\beta)\Big\vert_{\alpha=\beta=r}
			&=X\M w(1)\frac {r_K |U_K|^2a(2)}{\zeta_K(2)} \lz\frac{(\zeta^{(2)}_K(1+2r))'}{\zeta^{(2)}_K(1+2r)}+\sum_{(\varpi, 2)=1}\frac{\log N(\varpi)}{N(\varpi)(N(\varpi)^{1+2r}-1)}\pz.
\end{split}
\end{align}
	For the second term, we notice that due to the factor $1/\zeta_K(1-\alpha+\beta)$, only one term remains.  Moreover, $1/\zeta_K(s)=r^{-1}_K(s-1)+O((s-1)^2)$ and $P(1-\alpha, 1+\beta)=\zeta^{(2)}_K(2)/\zeta^{(2)}_K(2-2r)$. It follows that
\begin{align}
\label{M2der}
\begin{split}
			\frac{\dif}{\dif \alpha} &X^{1-\alpha}M_2(\alpha,\beta)\Big\vert_{\alpha=\beta=r}=-X^{1-r}\M w(1-r)(|D_K|N(c_K))^{-r}(2\pi)^{2r}\frac {\Gamma(\tfrac{1}{2}-r)}{\Gamma (\tfrac{1}{2}+r)}\frac {|U_K|^2a(2)}{\zeta^{(2)}_K(2-2r)} .
\end{split}
\end{align}
	
Next, as $R(X,\alpha,\beta)$ is analytic in $\alpha$, Cauchy's integral formula yields
\begin{align*}
\begin{split}
		\frac{\dif}{\dif \alpha}R(X,\alpha,\beta)=\frac{1}{2\pi i}\int\limits_{C_\alpha}\frac{R(X,z,\beta)}{(z-\alpha)^2} \dif z,
\end{split}
\end{align*}
   where $C_\alpha$ is a circle centered at $\alpha$ of radius $\rho$ with $\varepsilon/2<\rho<\varepsilon$. It follows that
\begin{align*}
\begin{split}
		\lab\frac{\dif}{\dif \alpha}R(X,\alpha,\beta)\rab\ll \frac1{\rho}\cdot\max_{z\in C_\alpha} |R(X,z,\beta)|\ll (1+|\alpha|)^{5/2+\varepsilon}(1+|\beta|)^{\varepsilon}X^{E(\alpha,\beta)+\varepsilon}.
\end{split}
\end{align*}

    We now set $\alpha=\beta=r$ to deduce from \eqref{Nab} and the above that
\begin{align}
\label{Rder1}
\begin{split}
		\lab\frac{\dif}{\dif \alpha}R(X,\alpha,\beta)\rab\ll (1+|r|)^{5/2+\varepsilon}X^{1-2r+\varepsilon}.
\end{split}
\end{align}
	
   We conclude from \eqref{Asymptoticshortversion}--\eqref{Rder1} that \eqref{Sum of L'/L with removed 2-factors} holds, completing the proof of Theorem \ref{Theorem for log derivatives}.

\vspace*{.5cm}

\noindent{\bf Acknowledgments.}   P. G. is supported in part by NSFC grant 11871082 and L. Z. by the FRG Grant PS43707 at the University of New South Wales.

\bibliography{biblio}
\bibliographystyle{amsxport}

\end{document}